\documentclass[draft]{amsart}

\usepackage{amssymb}
\usepackage{url}
\usepackage{amscd}
\theoremstyle{plain}
\newtheorem{theorem}{Theorem}[section]

\newtheorem{corollary}[theorem]{Corollary}
\newtheorem{lemma}[theorem]{Lemma}
\newtheorem{proposition}[theorem]{Proposition}
\newtheorem{definition}[theorem]{Definition}
\theoremstyle{remark}
\newtheorem{remark}[theorem]{Remark}
\newtheorem{example}[theorem]{Example}

\begin{document}

%\date{ }

\title{About the bound of the $\text{C}^*$ exponential length}

\author{Qingfei Pan and Kun Wang}

\address{Qingfei Pan: Department of Applied Mathematics\\
Sanming University\\
Sanming, Fujian\\
China 365004}

\email{pqf101@yahoo.com.cn}

\address{Kun Wang:Department of Mathematics\\
University of Puerto Rico, Rio
Piedras Campus, P.O.Box 70377\\
San Juan, Puerto Rico, USA 00931
}
\email{lingling-1106@hotmail.com}
%\thanks{Partially supported by grant DGICYT PB98-0618}


\begin{abstract}
Let $X$ be a compact Hausdorff space. In this paper, we give an example to show that there is $u\in \text{C}(X)\otimes \text{M}_n$ with $\det (u(x))=1$ for all $x\in X$ and $u\sim_h 1$ such that the $\text{C}^*$ exponential length of $u$
(denoted by $cel(u)$) can not be controlled by
$\pi$. This example answers a question of N.C. Phillips (see \cite{10}). Moreover, in simple inductive limit $\text{C}^*$-algebras, similar examples also exist.

\end{abstract}

\maketitle

%%%%%%%%%%%%%%%%%%%%%%%%%%%%%%%%%%%%%%%%%%%%%%%%%%%%%%%%%%%%%%%%%%%%%%%%%
% Macros
%%%%%%%%%%%%%%%%%%%%%%%%%%%%%%%%%%%%%%%%%%%%%%%%%%%%%%%%%%%%%%%%%%%%%%%%%

\newcommand\sfrac[2]{{#1/#2}}
\newcommand\cont{\operatorname{cont}}
\newcommand\diff{\operatorname{diff}}

%%%%%%%%%%%%%%%%%%%%%%%%%%%%%%%%%%%%%%%%%%%%%%%%%%%%%%%%%%%%%%%%%%%%%%%%%

\section{Introduction}
Exponential rank was introduced by Phillips and Ringrose \cite{PR}; and, subsequently, exponential length was introduced by Ringrose \cite{12}.
These invariants have been fundamental in the structure and classification of $\text{C}^*$-algebras. Among other things, they have played important roles in factorization and approximation properties for $\text{C}^*$-algebras e.g., the weak FU property \cite{6}, Weyl-von Neumann Theorems \cite{l1}, \cite{l2} (which in turn have been important in various generalizations of BDF Theory beyond the Calkin algebra case), and the uniqueness theorems of classification theory \cite{EGL}, \cite{l3}.

The $\text{C}^*$ exponential length and rank have  been extensively studied (see \cite{12}, \cite{PR}, \cite{6}, \cite{P}, \cite{Z1}, \cite{Z3}, \cite{Z2}, \cite{GL}, \cite{10}, \cite{8}, \cite{lin1}, \cite{11}, \cite{T}, \cite{lin2}, etc. -an incomplete list). In \cite{10} page 851, the paragraphs before Proposition 7.9, N.C. Phillips mentioned that
"We believe that suitable modifications of Lemma 5.2 and 5.3 will show that if $u\sim_h 1$ and $\det(u) = 1$,
then $cel(u) \leq \pi$ (even though, for general $u$, $cel(u)$ can be arbitrarily large)." However, in fact  for any $\varepsilon>0$,
we can find a unitary $u$ in a $\text{C}^*$-algebra with $\det(u)=1$, $u\sim_h 1$ and $cel (u)\geq 2\pi-\varepsilon$. In this paper, we provide a method for constructing such examples. In simple inductive limit $\text{C}^*$-algebras (simple AH algebras), we also get such examples. A recent paper of H. Lin's also proves similar results independently (not completely same) (see \cite{lin2}) but with different method. Note that for unital real rank zero   $\text{C}^*$-algebras, $\pi$ is an upper bound for the $\text{C}^*$ exponential length (see \cite{lin1}).
In a forthcoming paper, the second author will show that $2\pi$ is an upper bound for the $\text{C}^*$ exponential length of $\text{M}_n(C(X))$ provided the matrix size $n$ is large. Consequently, for the simple AH algebras $A$ (with no dimension growth) classified by Elliott-Gong-Li, the $\text{C}^*$ exponential length of elements in $CU(A)$---the closure of the commutator subgroup of $U(A)$, is at most $2\pi.$

\textbf{Acknowledgments} The authors would like to thank the referee for giving many valuable and constructive suggestions.
%%%%%%%%%%%%%%%

\section{Preliminaries}

For convenience of the reader, we recall some definitions and lemmas (see \cite{5}, \cite{12} for more details).

\begin{definition}
Let $X$ be a compact space and $B=\text{C}(X)\otimes \text{M}_n$, for $u\in U(B)$(unitary group of $B$), let $\det(u)$ be a function from $X$ to $S^1$ whose value at $x$ is $\det(u(x))$.
\end{definition}

%\begin{proposition} Let $u$ be a unitary in a $\text{C}^*$-algebra $A$ and suppose
%$u\sim_{h}1,$ then $$u=\text{exp}(ih_{1})\text{exp}(ih_{2})...\text{exp}(ih_{k}),$$
%for some self-adjoint elements, $h_i$, in $A$ and some integer $k$.
%\end{proposition}
\begin{definition}
 Let $A$ be a unital $\text{C}^*$-algebra and  $u$ be a unitary which lies in the connected component of the identity 1 in $A$. Define the $\text{C}^*$ exponential length of $u$ (denoted by $cel(u)$) as follows:
$$cel(u)=\inf\{\sum_{i=1}^{k}\|h_{i}\|: u=\text{exp}(ih_{1})\text{exp}(ih_{2})...\text{exp}(ih_{k})\}.$$
\end{definition}
\begin{remark} For a unital $\text{C}^*$-algebra $A$, let $U_0(A)$ be the connected component of $U(A)$ containing the identity $1$. Recall from \cite{12} that if $u\in U_0(A)$, then the $\text{C}^*$ exponential length $cel(u)$ is equal to the infimum of the lengths of rectifiable paths from $u$ to $1$ in $U(A)$.
\end{remark}
The following lemma is an easy example for calculating the $\text{C}^*$ exponential length.
\begin{lemma}\label{ex0} Let $\alpha\in\mathbb{R}$ and $u\in \text{C}[0,1]$ be defined by $u(t)=\exp{(it\alpha)}$, then $$cel(u)=\min_{k\in\mathbb{Z}}\max_{t\in[0,1]}|\alpha t-2k\pi|.$$ Moreover, if $|\alpha|\leq 2\pi$, then $cel(u)=|\alpha|$.
\end{lemma}
\begin{proof} Since $cel(u)=\inf\{ \text{length}(u_{s}): u_{s}\mbox{ is a path in }U(\text{C}[0,1]) \mbox{ from } u \mbox{ to } 1\}$, let $v_{s}(t)$ be any path from $1$ to $u$, that is,
$v_{0}(t)=1,v_{1}(t)=u(t).$ Without loss of generality, we can assume $v_s$ is piecewise smooth. Then $\text{length}(v_{s})=\int_{0}^{1}\|\frac{dv}{ds}\|ds.$ Since $v_{s}(t)$ can be considered as a map from $[0,1]\times[0,1]$ to $S^{1}$ and $\mathbb{R}$ is a covering space of $S^{1},$ there exists a unique map $\tilde{v}_s(t)$ from $[0,1]\times[0,1]$ to $\mathbb{R}$ such that:
$$v_s(t)=\pi(\tilde{v}_s(t))~~\mbox{and}~~\widetilde{v}_0(0)=0,\qquad (*)$$ where $\pi(x)=e^{ix}.$ Therefore
$$\frac{dv}{ds}=\pi'(\tilde{v}_s(t))\cdot\frac{d\widetilde{v}}{ds},$$ which implies $\|\frac{dv}{ds}\|=\|\frac{d\tilde{v}}{ds}\|.$

By $(*)$, $\pi(\tilde{v}_0(t))=v_0(t)=1$, hence $\tilde{v}_0(t)\in 2\pi\mathbb{Z}$ for all $t\in[0,1]$. Since $\tilde{v}_0(0)=0$ and $\tilde{v}_0(t)$ is continuous,  $\tilde{v}_0(t)=0$ for all $t\in[0,1]$. In addition, by $(*)$, we can also get $\pi(\tilde{v}_1(t))=v_1(t)=\exp(it\alpha)$,  thus $\tilde{v}_1(t)-\alpha t\in 2\pi \mathbb{Z}$ for all $t$. By continuity of $\tilde{v}_1(t)-\alpha t$, there exists some integer $k$ such that $\tilde{v}_1(t)-\alpha t=2k\pi$ for all $t\in[0,1]$. Therefore, $$\int_0^1\|\frac{d\tilde{v}}{ds}\|ds\geq \|\int_{0}^{1}\frac{d\tilde{v}}{ds}ds\|=\max_{t\in[0,1]}|\tilde{v}_1(t)- \tilde{v}_0(t)|\geq\min_{k\in\mathbb{Z}}\max_{t\in[0,1]}|\alpha t-2k\pi|.$$\\
Let $L=\min\limits_{k\in\mathbb{Z}}\max\limits_{t\in[0,1]}|\alpha t-2k\pi|$ and  $k_0\in\mathbb{Z}$ be such that $L=\max\limits_{t\in[0,1]}|\alpha t-2k_0\pi|$.
Fix $$v_s(t)=\exp\{is(\alpha t-2k_0\pi)\},$$ then $v_0(t)=\exp\{0\}=1$ and $v_1(t)=\exp\{i\alpha t-2k_0\pi i\}=\exp\{i\alpha t\}$ and $$\int_{0}^{1}||\frac{d{v}}{ds}||ds=\int_{0}^{1}||\alpha t -2k_0\pi||ds=\int_{0}^{1}\max\limits_{t\in[0,1]}|\alpha t-2k_0\pi|ds=\int_{0}^{1}Lds=L.$$ Thus $v_s(t)$ is a path in $U(C[0,1])$ connecting 1 and $u(t)$ with length $L$. Therefore, $cel(u)=L.$

Let us assume $\alpha\leq 2\pi$. For $k=0,$ $$\max\limits_{t\in[0,1]}|\alpha t-2k\pi|=\max\limits_{t\in[0,1]}|\alpha t-0|=|\alpha|.$$ For $k\neq 0$,
$$\max\limits_{t\in[0,1]}|\alpha t-2k\pi|\geq|0-2k\pi|=2|k|\pi\geq |\alpha|.$$ Hence, $\min\limits_{k\in\mathbb{Z}}\max\limits_{t\in[0,1]}|\alpha t-2k\pi|=|\alpha|.$ That is $cel(u)=|\alpha|$.
\end{proof}

\section{Counter-examples}

\begin{lemma}\label{deri}
Let $f(s,t):X\triangleq[0,1]\times[0,1]\rightarrow U(\text{M}_n(\mathbb{C}))$ be a smooth map. For any $\delta>0$, there is a smooth map $g(s,t):[0,1]\times[0,1]\rightarrow U(\text{M}_n(\mathbb{C}))$ such that:\\
(1) $\|f-g\|<\delta$, $\|\frac{\partial f}{\partial s}(s,t)-\frac{\partial g}{\partial s}(s,t)\|<\delta$, $\|\frac{\partial f}{\partial t}(s,t)-\frac{\partial g}{\partial t}(s,t)\|<\delta$;\\
(2) $g(s,t)$ has no repeated eigenvalues for all $(s,t)\in[0,1]\times[0,1]$.
\end{lemma}

\begin{proof}

This is the standard transversal theorem. Even though in the original statement in \cite{4}, it does not assert  that the derivatives are also close. But the proof really shows that
 (see pages 70-71 of \cite{4}). For convenience of the reader, we repeat the construction  here for our special case.

 By smoothly extending $f$ to an open neighborhood of $[0,1]\times[0,1]$, we can assume $f$ is defined on an open manifold without boundary. Let $Z$ be a subspace of $U(\text{M}_n(\mathbb{C}))$ defined by $$Z=\{u\in U(\text{M}_n(\mathbb{C})):~u\mbox{ has repeated eigenvalues}\}.$$
 Since $U(\text{M}_n(\mathbb{C}))$ is a subspace of $\text{M}_n(\mathbb{C})$ and the latter can be identified with $\mathbb{R}^{2n^2}$ as a topological space, $f$ is a smooth map from $X$ to $\mathbb{R}^{2n^2}$.
 Let $B$ be the open unit ball of $\mathbb{R}^{2n^2}$ (with Euclidean metric), then $B$ corresponds to some open ball (contained in the unit ball of $\text{M}_n(\mathbb{C})$) in $\text{M}_n(\mathbb{C})$ with the matrix norm, for which we still use the notation $B$.
 Let $0<\varepsilon<1/2$, for $x\in X$, $r\in B$, define $$F(x,r)=\pi[f(x)+\varepsilon r],$$ where $\pi:Gl_n(\mathbb{C})\rightarrow U(\text{M}_n(\mathbb{C}))$ is defined by the Polar decomposition, which serves as the map $\pi$ in \cite{4} from the tubular neighborhood of $U(\text{M}_n(\mathbb{C}))$ (which is $Y$ in the notation of \cite{4} page 70) to $U(\text{M}_n(\mathbb{C}))$. By the definition, $\pi$ is a smooth map.
 Define $f_r:X\rightarrow U(\text{M}_n(\mathbb{C}))$ by $$f_r(x)=F(x,r).$$ Since $\pi$ restricts to the identity on $U(\text{M}_n(\mathbb{C}))$, $$f_0(x)=F(x,0)=\pi(f(x))=f(x).$$

 For fixed $x$, $r\rightarrow f(x)+\varepsilon r$ is certainly a submersion of $B\rightarrow \text{M}_n(\mathbb{C})$. As the composition of two submersions is another, $r\rightarrow F(x,r)$ is a submersion. Therefore, $F$ is transversal to $Z$. Then the  Transversality Theorem (see page 68 of \cite{4}) implies that $f_r$ is transversal to $Z$ for almost all $r\in B$.

By a result of M. Choi and G. Elliott (see the second paragraph on page 77 of \cite{MG}), $Z$ is a finite union of embedded submanifolds of codimension at least three. Thus $$\dim(X)+\dim(Z)<\dim(U(\text{M}_n(\mathbb{C}))).$$ Therefore, $f_r$ transversal to $Z$ implies $\text{Im}f_r\bigcap Z=\emptyset.$

Since $f+\varepsilon r$ is in a neighborhood of $f$ and $\pi$ restricts to the identity on $U(\text{M}_n(\mathbb{C}))$, there exist positive real numbers $\varepsilon_1=\varepsilon_1(\varepsilon,r),~\varepsilon_2=\varepsilon_2(\varepsilon,r)$ such that:
 $$\|\frac{\partial f_r}{\partial s}\|\leq (1+\varepsilon_1)\|\frac{\partial f}{\partial s}\|,\quad\|\frac{\partial f_r}{\partial t}\|\leq(1+\varepsilon_2)\|\frac{\partial f}{\partial t}\|.$$
Therefore, by taking $r$ appropriately, we can get $f_r$ satisfies the properties (1) and (2). Finally, let $g=f_r$ and this completes the proof.

\end{proof}

\begin{remark} In \cite{6} Lemma 2.5, N. C. Phillips proves a similar result except for the property of derivatives.
\end{remark}

\begin{corollary}\label{leng}
Let $\widetilde{F}_s$ be a path in $U(\text{M}_k(C[0,1]))$. For any $\varepsilon>0$, there exists a path $F_s$ in $U(\text{M}_k(C[0,1]))$ such that:\\
(1) $\|F-\widetilde{F}\|<\varepsilon$;\\
(2) $F_s(t)$ has no repeated eigenvalues for all $(s,t)\in[0,1]\times[0,1]$;\\
(3) $|\text{length}(\widetilde{F})-\text{length}(F)|<\varepsilon$.\\
Moreover, if for each $t\in[0,1]$, $\widetilde{F}_1(t)$ has no repeated eigenvalues, then $F$ can be chosen to be such that $F_1(t)=\widetilde{F}_1(t)$ for all $t\in[0,1]$.
\end{corollary}

\begin{proof}
%Since in a manifold the length of $\widetilde{F}_s$ is defined by the supremum of sums of lengths of piecewise geodesic curves, we can assume $\widetilde{F}_s$ is piecewise  smooth. By doing some locally modification around  the (finitely many) non smooth points, we can make $\widetilde{F}_s$ to be a smooth path with length changed very small. Therefore, there exists a smooth path $\widehat{F}_s$ in   $U(\text{M}_n(C[0,1]))$ such that $$|\text{length}(\widetilde{F})-\text{length}(\widehat{F})|<\delta/2$$ and $\|\widetilde{F}-\widehat{F}\|<\delta/2.$
Let $\varepsilon_1$ be a small number to be determined later. Let $\delta_0$ be such that $|1-e^{i\theta}|\leq \delta_0$ implies $|\theta|\leq (1+\varepsilon_1)|1-e^{i\theta}|$ for $\theta\in\mathbb{R}$.
For $\varepsilon>0$, let $\delta=\min\{\delta_0,\varepsilon/6,1/2\}$.  By the definition of the length, there exist $0=s_0<s_1<s_2<\cdots<s_n=1$ such that $$\|\widetilde{F}_{s_{j+1}}-\widetilde{F}_{s_{j}}\|<\delta/2,\mbox{ for }j=0,1,\cdots, n-1,$$ and $$\sum_{j=0}^{n-1}\|\widetilde{F}_{s_{j+1}}-\widetilde{F}_{s_{j}}\|\leq length(\widetilde{F}_s)\leq\sum_{j=0}^{n-1}\|\widetilde{F}_{s_{j+1}}-\widetilde{F}_{s_{j}}\|+\varepsilon/4.$$
Note that for each $j$, $\widetilde{F}_{s_j}(t)$ is a continuous map from $[0,1]$ to $U(\text{M}_k(\mathbb{C}))$. There exist smooth maps $G_{s_j}(t):[0,1]\rightarrow U(\text{M}_k(\mathbb{C}))$,  such that $$\|G_{s_j}-\widetilde{F}_{s_j}\|=\sup_{t\in [0,1]}\|G_{s_j}(t)-\widetilde{F}_{s_j}(t)\|<\frac{\delta}{4n},~~~\quad~~ j=0,1,\cdots,n.$$
Then
\begin{align*}
\|G_{s_{j+1}}-G_{s_j}\|&=\|G_{s_{j+1}}-\widetilde{F}_{s_{j+1}}+\widetilde{F}_{s_{j+1}}-\widetilde{F}_{s_{j}}+\widetilde{F}_{s_{j}}-G_{s_j}\|\\&\leq \|\widetilde{F}_{s_{j+1}}-\widetilde{F}_{s_{j}}\|+\frac{\delta}{2n}\leq\delta.
\end{align*}
And by the first equality, we can also get $$\|G_{s_{j+1}}-G_{s_j}\|\geq \|\widetilde{F}_{s_{j+1}}-\widetilde{F}_{s_{j}}\|-\frac{\delta}{2n}.$$
Therefore, $$|\|G_{s_{j+1}}-G_{s_j}\|-\|\widetilde{F}_{s_{j+1}}-\widetilde{F}_{s_{j}}\||\leq\frac{\delta}{2n},$$
$$|\sum_{j=0}^{n-1}\|G_{s_{j+1}}-G_{s_j}\|-\sum_{j=0}^{n-1}\|\widetilde{F}_{s_{j+1}}-\widetilde{F}_{s_{j}}\||\leq\frac{\delta}{2}.$$
Thus $$\sum_{j=0}^{n-1}\|G_{s_{j+1}}-G_{s_j}\|-\frac{\delta}{2}\leq\mbox{length}(\widetilde{F}_s)\leq\sum_{j=0}^{n-1}\|G_{s_{j+1}}-G_{s_j}\|+\frac{\delta}{2}+\varepsilon/4,$$
$$\sum_{j=0}^{n-1}\|G_{s_{j+1}}-G_{s_j}\|-\frac{\varepsilon}{8}\leq\mbox{length}(\widetilde{F}_s)
\leq\sum_{j=0}^{n-1}\|G_{s_{j+1}}-G_{s_j}\|+\varepsilon/2.\quad (*)$$

Now we want to define a smooth function $$\widetilde{G}_s(t):[0,1]\times[0,1]\rightarrow U(\text{M}_k(\mathbb{C}))$$ such that $\widetilde{G}_{s_j}(t)=G_{s_j}(t)\mbox{ for }j=0,1,\cdots,n,~ t\in[0,1].$ We will define it piece by piece on each subinterval $[s_j,s_{j+1}]$ ($j=0,1,\cdots,n-1$).

Suppose $$
G^*_{s_j}G_{s_{j+1}}=U_j
  \begin{pmatrix}
   e^{i\alpha_1(t)} & 0 & \cdots & 0 \\
   0 & e^{i\alpha_2(t)} & \cdots & 0 \\
   \vdots  & \vdots  & \ddots & \vdots  \\
   0 & 0 & \cdots & e^{i\alpha_k(t)}
  \end{pmatrix}U_j^*.
$$
Since $\|G^*_{s_j}G_{s_{j+1}}-I\|=\|G_{s_j}-G_{s_{j+1}}\|\leq\delta<1$, there exists a self-adjoint element $H_j(t)$ in $\text{M}_k(C[0,1])$ such that
$G^*_{s_j}G_{s_{j+1}}(t)=e^{iH_j(t)}$, (here $H_j(t)=-i\log [G_{s_j}^*(t)G_{s_{j+1}}(t)]$ which is a smooth function). Define $$\widetilde{G}_s(t)=G_{s_j}(t)e^{i\frac{s-s_j}{s_{j+1}-s_j}H_j}\mbox{ for }s_j\leq s\leq s_{j+1}, ~t\in[0,1],~j=0,1,\cdots,n-1.$$
Then $\widetilde{G}_s(t)~(s_j\leq s\leq s_{j+1})$ is a path in $U(\text{M}_k(C[0,1]))$ from $G_{s_j}$ to $G_{s_{j+1}}$ and $\widetilde{G}_s(t)$ is smooth for $(s,t)\in [s_j,s_{j+1}]\times[0,1]$. Moreover,
\begin{align*}
&\mbox{length}(\widetilde{G}_s|_{s_j\leq s\leq s_{j+1}})&\\=&\int_{s_j}^{s_{j+1}}\|\frac{\partial \widetilde{G}_s}{\partial s}\|ds\\
\leq&\int_{s_j}^{s_{j+1}}\frac{1}{s_{j+1}-s_j}\|G_{s_j}(t)H_j(t)\|ds\\
\leq&(1+\varepsilon_1)\|G_{s_{j}}U_j \begin{pmatrix}
   1-e^{i\alpha_1(t)} & 0 & \cdots & 0 \\
   0 & 1-e^{i\alpha_2(t)} & \cdots & 0 \\
   \vdots  & \vdots  & \ddots & \vdots  \\
   0 & 0 & \cdots &1-e^{i\alpha_k(t)}
  \end{pmatrix}U_j^*\|\\
  =&(1+\varepsilon_1)\|G_{s_{j}} [I-U_j\begin{pmatrix}
  e^{i\alpha_1(t)} & 0 & \cdots & 0 \\
   0 & e^{i\alpha_2(t)} & \cdots & 0 \\
   \vdots  & \vdots  & \ddots & \vdots  \\
   0 & 0 & \cdots & e^{i\alpha_k(t)}
  \end{pmatrix}U_j^*]\|\\
= &(1+\varepsilon_1) \|G_{s_{j}}[I-G_{s_{j}}^*{G}_{s_{j+1}}]\|\\
=&(1+\varepsilon_1) \|G_{s_{j}}-{G}_{s_{j+1}}\|.
\end{align*}

Therefore, $\widetilde{G}_s$ $(0\leq s\leq 1)$ is a piecewise smooth path in $U(\text{M}_k(C[0,1]))$ and
$$\sum_{j=0}^{n-1}\|G_{s_{j+1}}-G_{s_j}\|\leq\mbox{length}(\widetilde{G}_s)\leq(1+\varepsilon_1)\sum_{j=0}^{n-1}\|G_{s_{j+1}}-G_{s_j}\|.$$
Thus by $(*)$, we have
$$\mbox{length}(\widetilde{G}_s)(1+\varepsilon_1)^{-1}-\frac{\varepsilon}{8}\leq \mbox{length}(\widetilde{F}_s)\leq \mbox{length}(\widetilde{G}_s)+\varepsilon/2.$$

Finally, pick any smooth monotone function $\xi:[0,1]\rightarrow [0,1]$ with $$\xi(0)=0,~\xi(1)=1,~\frac{d^n\xi}{ds^n}\mid_{s=0}=0,~\frac{d^n\xi}{ds^n}\mid_{s=1}=0~\mbox{ for all }n\geq 1.$$
Let $$\widetilde{G}'_s(t)=G_{s_j}(t)e^{i\xi(\frac{s-s_j}{s_{j+1}-s_j})H_j}\mbox{ for }s_j\leq s\leq s_{j+1}, ~t\in[0,1],~j=0,1,\cdots,n-1.$$
Then $\widetilde{G}'_s(t)$ is smooth for all $(s,t)\in[0,1]\times[0,1]$ (since $\frac{\partial\widetilde{G}'_s(t)}{\partial s}|_{s=s_j}=0$ from both left and right for all $j=1,2,\cdots,n-1$) and $$\mbox{length}(\widetilde{G}'_s)=\mbox{length}(\widetilde{G}_s).$$
And for each $(s,t)\in[0,1]\times[0,1]$,
\begin{align*}
\|\widetilde{G}'_s(t)-\widetilde{F}_s(t)\|=&\|\widetilde{G}'_s(t)-\widetilde{G}'_{s_j}(t)+\widetilde{G}'_{s_j}(t)-\widetilde{F}_{s_j}(t)+\widetilde{F}_{s_j}(t)
-\widetilde{F}_s(t)\|\\
\leq&\|\widetilde{G}_{s_{j+1}}(t)-\widetilde{G}_{s_j}(t)\|+\|\widetilde{G}_{s_j}(t)-\widetilde{F}_{s_j}(t)\|+\|\widetilde{F}_{s_j}(t)
-\widetilde{F}_s(t)\|\\
\leq &\delta+\frac{\delta}{4n}+\frac{\varepsilon}{4}\leq\varepsilon/2,
\end{align*}
where $s_j$ satisfies $s_j\leq s\leq s_{j+1}$.

Thus, by choosing $\varepsilon_1$ appropriately, we have
$$|\mbox{length}(\widetilde{G}'_s)-\mbox{length}(\widetilde{F}_s)|<\varepsilon/2\mbox{ and }\|\widetilde{G}'-\widetilde{F}\|<\varepsilon/2.$$

  Since $\widetilde{G}'_s$ can be seen as a smooth map from $[0,1]\times[0,1]$ to $U(\text{M}_k(\mathbb{C}))$, by Lemma \ref{deri}, there exists $F$ such that $\|F-\widetilde{G}'\|<\varepsilon/2$ and $F_s(t)$ has no repeated eigenvalues for all $(s,t)\in[0,1]\times[0,1]$. Moreover,
$$|\text{length}(F_s)-\text{length}(\widetilde{G}'_s)|=|\int_0^1\|\frac{\partial F}{\partial s}\|ds-\int_0^1\|\frac{\partial \widetilde{G}'}{\partial s}\|ds|<\varepsilon/2.$$
Thus $F$ satisfies properties (1)-(3), which is what we want.

Moreover, if $\widetilde{F}_1(t)$ has no repeated eigenvalues for all $t\in[0,1]$, then there exists $\eta>0$ such that $\|u(t)-\widetilde{F}_1(t)\|<\eta$ implies $u(t)$ has no repeated eigenvalues for all $t\in [0,1]$. For $\varepsilon=\eta/2$, by the first part of the statement we can find a path $ F_s$ satisfying properties (1), (2), (3).
Let $s_0\in[s_{n-1},1)$ be such that $\|F_{s_0}(t)-F_1(t)\|\leq\eta/4$, (where $s_{n-1}$ is a point of the partition of $[0,1]$ for which  we mentioned in the proof of the first part of the statement). Then
$$\|F_{s_0}(t)-\widetilde{F}_1(t)\|=\|F_{s_0}(t)-F_1(t)+F_1(t)-\widetilde{F}_1(t)\|\leq3\eta/4.$$
Now let us redefine $F_s(t)$ on the subinterval $[s_0,1]$ (still use the notation $F_s(t)$) by a similar way as above:
$$F_s(t)=F_{s_0}(t)e^{i\frac{s-s_0}{1-s_0}H_n(t)},~~ \mbox{ for }s_0\leq s\leq1,$$ where $H_n(t)=-i\log [F_{s_0}^*(t)\widetilde{F}_1(t)]$.
%Then we can get that this newly defined path $F_s$ can be close enough to the old $F_s$ and their length are also close.
Since this newly defined path $F_s$ lies in the $\eta$ neighborhood of $F_1$, $F_s(t)$ has no repeated eigenvalues for all $(s,t)\in [0,1]\times[0,1]$.
Thus this $F_s$ is what we want.

\end{proof}

\begin{definition} For a metric space $(Y,d)$, let $$P^{k}Y:=\{(y_{1},y_{2},...,y_{k}): y_{i}\in Y\}/\sim,$$ where $(y_{1},y_{2},...,y_{k})\sim(\widetilde{y}_{1},\widetilde{y}_{2},...,\widetilde{y}_{k})$ if $\exists \sigma\in S_{k}$ such that $y_{\sigma(i)}=\tilde{y}_{i}\mbox{ for all }1\leq i\leq k$. Let $[y_{1},y_{2},...,y_{k}]$ denote the equivalent class of $(y_1,y_2,\cdots,y_k)$  in $P^{k}Y$. Define also the metric of $P^{k}Y$ as:
$$\text{dist}([y_{1},y_{2},...,y_{k}],[\widetilde{y}_{1},\widetilde{y}_{2},...,\widetilde{y}_{k}])=\min_{\sigma\in S_{k}}\max_{1\leq i\leq k}d(y_{i},\widetilde{y}_{\sigma(i)}).$$
\end{definition}
The proof of the following lemma is straightforward.
\begin{lemma}\label{cover} Let $(Y,d)$ be a metric space and $\pi: \underbrace{Y\times Y\times ...\times Y}_{k}\longrightarrow P^{k}Y$ be the quotient map. Let $X\subset\underbrace{Y\times Y\times ...\times Y}_{k}$ be the set consisting of those elements $(y_{1},y_{2},...,y_{k})$ with $y_{i}\neq y_{j}$ if $i\neq j$. Then the restriction of $\pi$ to $X$ is a covering map.
\end{lemma}
%\begin{proof} Let $[z_{1},z_{2},...,z_{k}]\in \pi(X)$ be a fixed point, by the definition of $X$, $z_{i}\neq z_{j}$ for $i\neq j$. Let $$\varepsilon=\frac{1}{2}\min\{d(z_{i},z_{j}): i\neq j, 1\leq i,j\leq k\},$$
%$$N_{\varepsilon}=\{[y_{1},y_{2},...,y_{k}]\in P^{k}Y: \text{dist}([y_{1},y_{2},...,y_{k}],[z_{1},z_{2},...,z_{k}])<\varepsilon\}.$$
%Claim: $N_{\varepsilon}\subseteq\pi(X).$\\
%In fact, if there is $[y_{1},y_{2},...,y_{k}]\in N_{\varepsilon}$ and $y_{i}=y_{j}$ for some $1\leq i,j\leq k$ and $i\neq j$, then $\exists\sigma\in S_{k}$ such that $d(y_{i},z_{\sigma(i)})<\varepsilon$ and $d(y_{j},z_{\sigma(j)})<\varepsilon.$ Therefore $$d(z_{\sigma(i)},z_{\sigma(j)})\leq d(y_{i},z_{\sigma(i)})+d(y_{j},z_{\sigma(j)})<2\varepsilon,$$ which contradicts to the definition of $\varepsilon.$

%It is easy to see that $N_{\varepsilon}$ is an open neighbourhood of $[z_{1},z_{2},...,z_{k}]$ and $\pi^{-1}(N_{\varepsilon})$ consists of $k!$ pairwise disjoint open subsets (denoted by $U_{i}, 1\leq i\leq k!$) of $X$, and the restriction of $\pi$ to each $U_{i}$ is a homeomorphism from $U_{i}$ to $N_{\varepsilon}.$ Hence, the restriction of $\pi$ on $X$ is a covering map.\end{proof}

We need the  following easy lemma.
\begin{lemma}\label{sm} Let $F: [0,1]\times[0,1]\rightarrow P^{k}S^{1}$ be a continuous function, suppose $$F(s,t)=[x_{1}(s,t),x_{2}(s,t),...,x_{k}(s,t)],$$ and for all $(s,t)\in [0,1]\times[0,1]$, $x_{i}(s,t)\neq x_j(s,t)$ if $i\neq j$. Then there are continuous functions $f_{1},f_{2},...,f_{k}:[0,1]\times[0,1]\rightarrow S^{1}$ such that: $$F(s,t)=[f_{1}(s,t),f_{2}(s,t),...,f_{k}(s,t)].$$
\end{lemma}
\begin{proof} Let $\pi: \underbrace{S^{1}\times S^{1}\times...\times S^{1}}_{k}\rightarrow P^{k}S^{1}$ denote the quotient map, and let $X\subset \underbrace{S^{1}\times S^{1}\times...\times S^{1}}_{k}$ be the set consisting of those elements $(x_1,x_2, \cdots ,x_k)$ with $x_i\neq x_j$ if $i\neq j$. Then by Lemma \ref{cover} $\pi|_X$ is a covering map  from $X$ to $\pi(X)$ (which is a subset of $P^kS^1$).

%The following paragraph is wrong, so should be deleted and replaced by what follows.

%Let $G\subset X$ such that $G$ is homeomorphic to $F([0,1]\times [0,1])$ and
%denote $\pi_{G}: G\rightarrow F([0,1]\times [0,1])$ the restriction of $\pi|_X$ on $G$.

Note from the assumption of the Lemma, the image of $F$ is contained in  $\pi(X)$. Since $[0,1]\times[0,1]$ is simply connected, by the standard lifting theorem for covering spaces, the map $F:[0,1]\times[0,1]\to \pi(X)\subset P^k S^1$
can be lifted to a map $F_1:[0,1]\times[0,1] \to X (\subset \underbrace{S^{1}\times S^{1}\times...\times S^{1}}_{k})$.

Let $\pi_{j}:S^{1}\times S^{1}\times...\times S^{1}\rightarrow S^{1}$ be the projection onto the $j$th coordinate. For $1\leq j\leq k,$ define functions $f_{j}: [0,1]\times [0,1]\rightarrow S^{1}$ by $$f_{j}(s,t)=\pi_{j}(F_1(s,t)).$$ Then it is easy to see that $f_{j}$'s satisfy the requirements.
\end{proof}
\begin{remark}\label{remark} Let $F_s$ be a path in $U(M_k(C[0,1]))$ such that $F_s(t)$ has no repeated eigenvalues for all $(s,t)\in[0,1]\times[0,1]$. Let $\Lambda:[0,1]\times[0,1]\rightarrow P_k S^1$ be the eigenvalue map of $F_s(t)$, i.e. $\Lambda(s,t)=[x_1(s,t),x_2(s,t),\cdots,x_k(s,t)],$ where $\{x_i(s,t)\}_{i=1}^k$ are eigenvalues of the matrix $F_s(t)$. By Lemma \ref{sm}, there are continuous functions $f_{1},f_{2},...,f_{k}:[0,1]\times[0,1]\rightarrow S^{1}$ such that: $$\Lambda(s,t)=[f_{1}(s,t),f_{2}(s,t),...,f_{k}(s,t)].$$ For each fixed $(s,t)\in[0,1]\times[0,1]$, there is a unitary $U_s(t)$ such that $$F_s(t)=U_s(t)diag[f_{1}(s,t),f_{2}(s,t),...,f_{k}(s,t)]U_s(t)^*.$$ Note that $U_s(t)$ can be chosen to be continuous, but in this paper we don't need this property.

\end{remark}

\begin{proposition}\label{we} (see line 13-line 18 of page 71 of \cite{1}) If $U,V\in M_n(\mathbb{C})$ are unitaries with eigenvalues $u_1,u_2,\cdots,u_n$ and $v_1,v_2,\cdots, v_n$ respectively, then $$\min\limits_{\sigma\in S_n}\max\limits_{i}|u_i-v_{\sigma(i)}|\leq\|U-V\|.$$
\end{proposition}
The same result for a pair of Hermitian matrices is due to H. Weyl (called Weyl's Inequality see \cite{14}).

\begin{lemma}\label{ca} Let $F_s$ be a path in $U(\text{M}_{n}(\text{C}[0,1]))$ and $f^{1}_s(t),f^{2}_s(t),...,f^{n}_s(t)$ be continuous functions such that $$F_s(t)=U_s(t)\text{diag}[f^{1}_s(t),f^{2}_s(t),...,f^{n}_s(t)]U_s(t)^*,$$ where $U_s(t)$ are unitaries. Suppose for any $(s,t)\in[0,1]\times[0,1]$, $f_s^i(t)\neq f_s^j(t)$ if $i\neq j$, then $$\text{length}(F_s)\geq\max_{1\leq j\leq n}\{\text{length}(f_s^j)\}.$$
(In this lemma, we assume that $F_s(t)$ is continuous, but we do not assume $U_s(t)$ is continuous.)
\end{lemma}
\begin{proof} Let $$\varepsilon=\min\{|(f_s^i(t)-f_s^j(t))|:~i\neq j,~1\leq i,j\leq n,~s\in[0,1],~t\in[0,1]\}.$$ Since for each $j$, $f_s^j(t)$ is continuous with respect to $s$, there exists $\delta>0$ such that: for any partition $\mathcal{P}=\{s_1,s_2,\cdots,s_{\lambda}\}$ with $|\mathcal{P}|<\delta$, $$\|f_{s_i}^j(t)-f_{s_{i-1}}^j(t)\|<\varepsilon/2~~\mbox{ for all }2\leq i\leq\lambda,~~1\leq j\leq n.$$
Then by Proposition \ref{we},
\begin{align*}
\text{length}(F_s)\geq &\sum\limits_{i=2}^{\lambda}\|F_{s_i}-F_{s_{i-1}}\|=\sum\limits_{i=2}^{\lambda}\sup_{t\in[0,1]}\|F_{s_i}(t)-F_{s_{i-1}}(t)\|\\
\geq &\sum\limits_{i=2}^{\lambda}\sup_{t\in[0,1]}[\min\limits_{\sigma\in S_n}\max\limits_{1\leq j\leq n}|f_{s_i}^j(t)-f_{s_{i-1}}^{\sigma(j)}(t)|].
\end{align*}
If $\sigma(j)\neq j$, then
\begin{align*}
&|f_{s_i}^j(t)-f_{s_{i-1}}^{\sigma(j)}(t)|\\ \geq &|f_{s_i}^j(t)-f_{s_{i}}^{\sigma(j)}(t)|-|f_{s_{i}}^{\sigma(j)}(t)-f_{s_{i-1}}^{\sigma(j)}(t)|\\
>&\varepsilon-\varepsilon/2=\varepsilon/2.
 \end{align*}
 If $\sigma(j)=j$, then $|f_{s_i}^j(t)-f_{s_{i-1}}^j(t)|<\varepsilon/2$. Therefore,
$$\min\limits_{\sigma\in S_n}\max\limits_{1\leq j\leq n}|f_{s_i}^j(t)-f_{s_{i-1}}^{\sigma(j)}(t)|=\max\limits_{1\leq j\leq n}|f_{s_i}^j(t)-f_{s_{i-1}}^j(t)|.$$  \begin{align*}
\text{length}(F_s)&\geq \sum\limits_{i=2}^{\lambda}\sup_{t\in[0,1]}[\min\limits_{\sigma\in S_n}\max\limits_{1\leq j\leq n}|f_{s_i}^j(t)-f_{s_{i-1}}^{\sigma(j)}(t)|]\\
&\geq\sum\limits_{i=2}^{\lambda}\sup_{t\in[0,1]}\max\limits_{1\leq j\leq n}|f_{s_i}^j(t)-f_{s_{i-1}}^j(t)|\\
 &=\sum\limits_{i=2}^{\lambda}\max\limits_{1\leq j\leq n}\|f_{s_i}^j(t)-f_{s_{i-1}}^j(t)\|.~~~~(*)
 \end{align*}
Since $(*)$ holds for any partition $\mathcal{P}$ with $|\mathcal{P}|<\delta$, we have $$\text{length}(F_s)\geq\max\limits_{1\leq j\leq n}\text{length}(f_{s}^j).$$
\end{proof}
\begin{example}\label{Ex1} Let $A=\text{M}_{10}(\text{C}[0,1])$. Define
$$
u(t) =
  \begin{pmatrix}
   e^{-2\pi it\frac{9}{10}} & 0 & \cdots & 0 \\
   0 & e^{2\pi it\frac{1}{10}} & \cdots & 0 \\
   \vdots  & \vdots  & \ddots & \vdots  \\
   0 & 0 & \cdots & e^{2\pi it\frac{1}{10}}
  \end{pmatrix}_{10\times 10}.
$$
Then $u$ is a unitary in $A$ with $\det(u)=1$ and $u\sim_{h}1.$
\end{example}

\begin{theorem} Let $u\in \text{M}_{10}(\text{C}[0,1])$ be defined as in \ref{Ex1}, then $$cel(u)\geq 2\pi\cdot\frac{9}{10}.$$\end{theorem}
\begin{proof} Let $\widetilde{F}_s$ be a path in $U(\text{M}_{10}(\text{C}[0,1]))$ with $\widetilde{F}_0=1\in \text{M}_{10}(\text{C}[0,1])$ and $\widetilde{F}_1(t)=u(t).$ By Corollary \ref{leng}, there is a path $F_s$ in
$U(\text{M}_{10}(\text{C}[0,1]))$ such that (1) $\|F-\widetilde{F}\|<\varepsilon/2$, (2) $F_s(t)$ has no repeated eigenvalues for all $(s,t)\in[0,1]\times[0,1]$, and (3)$|\text{length}(\widetilde{F})-\text{length}(F)|<\varepsilon/2$. By Lemma \ref{sm} and Remark \ref{remark}, there are continuous maps $f^{1},f^{2},...,f^{10}:[0,1]\times[0,1]\rightarrow S^{1}$ and unitaries $U_s(t)$ such that $F_s(t)=U_s(t)\text{diag}[f^{1}_s(t),f^{2}_s(t),...,f^{10}_s(t)]U_s(t)^*$. By Lemma \ref{ca}, $$\text{length}(F_s)\geq\max_{1\leq i\leq 10}\{\text{length}(f_s^i)\}.$$
Since $\|F-\widetilde{F}\|<\varepsilon/2,$
 $\|f^j_0-1\|< \varepsilon/2 \mbox{ for all } 1\leq j\leq 10$. For each fixed  $t\in (\varepsilon, 1-\varepsilon)$,  there exists one and only one $j_0$ such that $$\|f^{j_0}_1(t)-e^{-2\pi it\frac{9}{10}}\|< \varepsilon/2.~~$$ For other $j\neq j_0$
$$\|f^{j}_1(t)-e^{ 2\pi it\frac{1 }{10}}\|< \varepsilon/2. $$
(We use the fact that $\|e^{ 2\pi it\frac{1}{10}}-e^{-2\pi it\frac{9}{10}}\|\geq \varepsilon$ for $t\in (\varepsilon, 1-\varepsilon)$ .)
Since all $f^j$'s are continuous, the index $j_0$ should be the same for all $t\in (\varepsilon, 1-\varepsilon)$.

Thus $f^{j_0}_s$ is a path in $U(\text{C}[0,1])$ connecting a point near 1 and a point near $e^{-2\pi it\frac{9}{10}}$.
  By Lemma \ref{ex0}, $\text{length}(f^{j_0}_s)\geq\frac{9}{10}\cdot2\pi-\varepsilon.$ Therefore,
$$\text{length}(\tilde{F}_s)\geq\frac{9}{10}\cdot2\pi-\varepsilon/2-\varepsilon\geq\frac{9}{10}\cdot2\pi-2\varepsilon.$$ Since $\varepsilon$ is arbitrary, we have $cel(u)\geq \frac{9}{10}\cdot 2\pi$, which completes the proof.
\end{proof}

\begin{example}\label{Ex2} Examples in some simple inductive limit $\text{C}^*$-algebras.

Let $\{x_{1}, x_{2}, ...\}$ be a countable distinct dense subset of $[0,1]$ and let
$\{k_{n}\}_{n=2}^{\infty}$ be a sequence of integers satisfying $$\prod_n\frac{10^{k_{n}}-1}{10^{k_{n}}}>\frac{11}{12}.$$
Let
$$A_{1}=\text{M}_{10}(\text{C}[0,1]),~ A_{2}=\text{M}_{10^{k_{2}}}(A_{1}), ~..., ~A_{n}=\text{M}_{10^{k_{n}}}(A_{n-1}),....$$ Define $\varphi_{n,n+1}: A_{n}\rightarrow A_{n+1}$ by

$$
\varphi_{n,n+1}(f) =
  \begin{pmatrix}
   f & 0 & \cdots & 0 &0\\
   0 & f & \cdots & 0 &0\\
   \vdots  & \vdots  & \ddots & \vdots &\vdots \\
   0 & 0 & \cdots & f &0\\
   0 & 0 & \cdots & 0 &f(x_{n})\\
  \end{pmatrix}_{10^{k_{n}}\times 10^{k_{n}}}
$$ and
 $A=\underrightarrow{\lim}(A_{i}, \varphi_{i,i+1})$ be the inductive limit $\text{C}^*$-algebra. Then $A$ is simple.\\
Let $u(t)\in A_1$ be defined as in Example \ref{Ex1}, then  (see Theorem \ref{th2} and Corollary \ref{co3}) $$cel(\varphi_{1,\infty}(u))\geq \frac{9}{10}\cdot 2\pi. $$
\end{example}

Inductive limit of such form was studied by Goodearl \cite{G} and its exponential rank was calculated by Gong and Lin \cite{GL}.

\begin{lemma}\label{isometry} Let $\theta:P^L\mathbb{R}\rightarrow (\mathbb{R}^L,d_{max})$ be defined by $$\theta[x_1,x_2,\cdots,x_L]=(y_1,y_2,\cdots,y_L)$$
iff $[x_1,x_2,\cdots,x_L]=[y_1,y_2,\cdots,y_L]$ and $y_1\leq y_2\leq\cdots\leq y_L$. Then $\theta$ is an isometry.
\end{lemma}
\begin{proof} Let $a=[a_1,a_2,\cdots,a_L],~b=[b_1,b_2,\cdots,b_L]$ be any two elements in $P^L\mathbb{R}$. Without loss of generality, we can assume that $a_1\leq a_2\leq\cdots\leq a_L$ and $b_1\leq b_2\leq\cdots\leq b_L$. Thus $$d_{max}(\theta(a),\theta(b))=\max_{1\leq i\leq L}|a_i-b_i|.$$
If $dist(a,b)\neq \max\limits_{1\leq i\leq L}|a_i-b_i|$, then there exist a permutation $\sigma\in S_L$ such that $$l\triangleq \max_{1\leq i\leq L}|a_i-b_{\sigma(i)}|<\max_{1\leq i\leq L}|a_i-b_i|.$$ Since $l<\max\limits_{1\leq i\leq L}|a_i-b_i|$, there exists $k$ such that $|a_k-b_k|>l.$\\
If $a_k<b_k$, then  $|a_i-b_j|>l$ for any $i\leq k,j\geq k.$ Since the cardinality of the set $\{\sigma(1),\sigma(2),\cdots,\sigma(k)\}$ is $k$, there is at least one element $i_0\in \{1,2,\cdots,k\}$ with $\sigma(i_0)\geq k$. Then  $|a_{i_0}-b_{\sigma(i_0)}|>l$. Therefore, $$\max_{1\leq i\leq k}|a_i-b_{\sigma(i)}|>l.$$\\
Similarly, if $a_k>b_k$, one can prove $$\max_{k\leq i\leq L}|a_i-b_{\sigma(i)}|>l.$$ In either case, it contradicts
$l=\max\limits_{1\leq i\leq L}|a_i-b_{\sigma(i)}|$. Therefore, $$dist(a,b)= \max\limits_{1\leq i\leq L}|a_i-b_i|=d_{max}(\theta(a),\theta(b)),$$ which means $\theta$ is an isometry.

\end{proof}

\begin{theorem}\label{th2} Let $A_i,~\varphi_{i,i+1},~(i\in \mathbb{N})$ be defined as in \ref{Ex2}, for any $\varepsilon\in (0,\frac{1}{100})$, let $u_\varepsilon\in A_1$ be defined by:
$$
u_\varepsilon(t) =
  \begin{pmatrix}
   e^{-2\pi it(\frac{9}{10}-\varepsilon)} & 0 & \cdots & 0 \\
   0 & e^{2\pi it(\frac{1}{10}-\varepsilon)} & \cdots & 0 \\
   \vdots  & \vdots  & \ddots & \vdots  \\
   0 & 0 & \cdots & e^{2\pi it(\frac{1}{10}-\varepsilon)}
  \end{pmatrix}_{10\times 10},
$$then $$cel(\varphi_{1,n}(u_\varepsilon))\geq 2\pi(\frac{9}{10}-\varepsilon)-5\varepsilon
~~\mbox{for all } n\in \mathbb{N}.$$
\end{theorem}
\begin{proof} By an easy calculation, we know that $$\varphi_{1,n}(u_\varepsilon)\in M_{L}(C[0,1]),~~\text{ where } L=10\prod_{i=2}^n 10^{k_i}.$$ On the diagonal of
$\varphi_{1,n}(u_\varepsilon)$, there are $\prod\limits_{i=2}^n (10^{k_i}-1)$ terms equal to $e^{-2\pi it(\frac{9}{10}-\varepsilon)}$, $9\cdot\prod\limits_{i=2}^n (10^{k_i}-1)$ terms equal to $e^{2\pi it(\frac{1}{10}-\varepsilon)}$ and the rest are constants. Let $\alpha$, $\beta$ and $\gamma$ denote, respectively, the numbers of terms of the forms $e^{-2\pi it(\frac{9}{10}-\varepsilon)}$, $e^{2\pi it(\frac{1}{10}-\varepsilon)}$ and constants on the diagonal of $\varphi_{1,n}(u_\varepsilon)$ (i.e. $\alpha=\prod\limits_{i=2}^n (10^{k_i}-1)$, $\beta=9\cdot\prod\limits_{i=2}^n (10^{k_i}-1)$ and $\gamma=L-\alpha-\beta$). Therefore $\frac{\alpha+\beta}{L}=\prod\limits_{i=2}^n\frac{10^{k_{i}}-1}{10^{k_{i}}}>\frac{11}{12}$, which implies $\frac{\alpha}{\gamma}>\frac{11}{10}$.

For each $t\in[0,1]$, let $E(t)$ be the set consisting of all eigenvalues of $\varphi_{1,n}(u_\varepsilon)(t)$ (counting multiplicities). Define continuous functions  $\bar{\bar{y}}_k(t)~(1\leq k\leq L)$ from $[0,1]$ to $[-\frac{9}{10}+\varepsilon,\frac{1}{10}-\varepsilon]$ as follows:
\begin{align*}
  &\bar{\bar{y}}_{k}(t)=-(\frac{9}{10}-\varepsilon)t,~~~~\mbox{if }1\leq k\leq\alpha, \\
  &\bar{\bar{y}}_{k}(t)=(\frac{1}{10}-\varepsilon)t,~~~~~~\mbox{if }\alpha+1\leq k\leq\alpha+\beta,\\
  &\bar{\bar{y}}_{k}(t)=-(\frac{9}{10}-\varepsilon)x_\bullet\mbox{ or }(\frac{1}{10}-\varepsilon)x_\bullet,~~~~\mbox{if }\alpha+\beta+1\leq k\leq L,
\end{align*}
 for $ x_\bullet\in\{x_1,x_2,\cdots,x_n\}$ { such that } $\{\exp\{2\pi i\bar{\bar{y}}_k(t)\}:1\leq k\leq L\}=E(t)$ for $\forall t\in[0,1]$.

Let $\theta$ be the map defined in Lemma \ref{isometry} and let $p_k$ $(1\leq k\leq L)$ be projections from $\mathbb{R}^L$ to $\mathbb{R}$ with respect to the $k$-th coordinate. Define functions $\bar{y}_{k}(t):[0,1]\rightarrow[-\frac{9}{10}+\varepsilon,\frac{1}{10}-\varepsilon]$ for $1\leq k\leq L$ as follows:
$$\bar{y}_k(t)=p_k\theta[\bar{\bar{y}}_1(t),\bar{\bar{y}}_2(t),\cdots,\bar{\bar{y}}_L(t)].$$
%For each fixed $t\in[0,1]$, let $E'(t)=\{\bar{\bar{y}}_k(t):1\leq k\leq L\}$ (counting multiplicities). Let
%\begin{align*}
%  &\bar{y}_{1}(t)=\min E'(t), \\
%  &\bar{y}_{2}(t)=\min\{E'(t)\setminus \{\bar{y}_{1}(t)\}\},\\
%  &\bar{y}_{3}(t)=\min\{E'(t)\setminus \{\bar{y}_{1}(t),\bar{y}_{2}(t)\}\},\\
%  &......
%\end{align*}
%where we count multiplicities for each set, (which means for example, if $E'(0)=\{-\frac{1}{4},-\frac{1}{4},-\frac{1}{4},0,\frac{1}{2}\}$, then $\{E'(0)\setminus \{\bar{y}_{1}(0)\}\}=\{-\frac{1}{4},-\frac{1}{4},0,\frac{1}{2}\}$).
Then $\bar{y}_{k}(t)$ $(1\leq k\leq L)$ are continuous functions with $\bar{y}_{1}(t)\leq \bar{y}_{2}(t)\leq\cdots\leq \bar{y}_{L}(t)$ for all $ t\in [0,1]$ and $\{\bar{y}_{1}(t),\bar{y}_{2}(t),\cdots,\bar{y}_{L}(t)\}$ (counting  multiplicities) is equal to $\{\bar{\bar{y}}_1(t),\bar{\bar{y}}_2(t),\cdots,\bar{\bar{y}}_L(t)\}$.
 For $t\in(0,1)$,  there are at most $\gamma$ terms (the constants referred to above) of  $\{\bar{y}_k(t)\}_{k=1}^L$ which could be less than $-(\frac{9}{10}-\varepsilon)t$, therefore,
$$\bar{y}_k(t)\geq -(\frac{9}{10}-\varepsilon)t,~~~\mbox{for }k\geq \gamma+1.$$ Similarly, for $t\in(0,1)$ there are at most $\gamma+\beta$ terms (the  constants or terms of the form $(\frac{1}{10}-\varepsilon)t$) of $\{\bar{y}_k(t)\}_{k=1}^{L}$ which could be greater than $-(\frac{9}{10}-\varepsilon)t$, so $$\bar{y}_k(t)\leq -(\frac{9}{10}-\varepsilon)t,~~~\mbox{for }k\leq L-\gamma-\beta=\alpha.$$
Since $\alpha>\gamma$,
$$ \bar{y}_k(t)=-(\frac{9}{10}-\varepsilon)t,~~\mbox{ for }\gamma+1\leq k\leq \alpha.$$
Let $y_{k}(t)=\exp\{2\pi i{\bar{y}_k(t)}\}$ for $1\leq k\leq L$. Then we have
$$ y_k(t)=e^{-2\pi it(\frac{9}{10}-\varepsilon)},~~\mbox{ for }\gamma+1\leq k\leq \alpha,$$
and $$\{y_{k}(t):1\leq k\leq L\}=E(t),~~ \mbox{ for all }t\in[0,1].$$

Let
$$
W(t) =
  \begin{pmatrix}
   y_{1}(t) & 0 & \cdots & 0 \\
   0 & y_{2}(t) & \cdots & 0 \\
   \vdots  & \vdots  & \ddots & \vdots  \\
   0 & 0 & \cdots & y_{L}(t)
  \end{pmatrix}.
$$

\noindent Then for all $t\in[0,1]$, $W(t)$ and $\varphi_{1,n}(u_\varepsilon)(t)$ have exactly same eigenvalues (counting multiplicities). By a result of Thomsen (see \cite{13} Theorem 1.5), $\exists \Lambda(t)\in U(M_{L}(C[0,1]))$ such that $$\|\Lambda(t)W(t)\Lambda(t)^{*}-\varphi_{1,n}(u_\varepsilon)(t)\|<\varepsilon,$$ for all $t\in[0,1].$ Therefore, $cel(\varphi_{1,n}(u_\varepsilon))\geq cel(W)-\varepsilon\pi/2>cel(W)-2\varepsilon$. (Here we use two facts: for unitaries $a,b$, (1) $\|a-b\|<\varepsilon<1$ implies $|cel(a)-cel(b)|<\varepsilon\pi/2$ and (2) $cel(uau^*)=cel(a)$ where $u$ is a unitary).

Let $-\varepsilon<\varepsilon_1<\varepsilon_2<\cdots<\varepsilon_{\alpha}=0<\varepsilon_{\alpha+1}<\varepsilon_{\alpha+2}<\cdots<\varepsilon_{L}<\varepsilon$
be chosen, and $\tilde{y}_k(t)=\exp\{2\pi i(\bar{y}_k(t)+\varepsilon_k)\}$ and let
$$
\widetilde{W}(t) =
  \begin{pmatrix}
   \widetilde{y}_{1}(t) & 0 & \cdots & 0 \\
   0 & \widetilde{y}_{2}(t) & \cdots & 0 \\
   \vdots  & \vdots  & \ddots & \vdots  \\
   0 & 0 & \cdots & \widetilde{y}_{L}(t)
  \end{pmatrix}\in U(M_{L}(C[0,1])).
$$
Then $\widetilde{W}(t)$ has no repeated eigenvalues for all $t\in[0,1]$ and
$$\|\widetilde{W}(t)-W(t)\|<\varepsilon,$$ $$\widetilde{y}_{i}(t)\neq\widetilde{y}_{j}(t) \mbox{ for } i\neq j\mbox{ and } \widetilde{y}_{{\alpha}}(t)=e^{-2\pi it(\frac{9}{10}-\varepsilon)}.$$

 Let $\widetilde{F}_{s}(t)$ be a path in $U(M_L(C[0,1]))$ from $1$ to $\widetilde{W}(t)$ with $\widetilde{F}_0(t)=1$ and $\widetilde{F}_1(t)=\widetilde{W}(t).$ By Corollary \ref{leng}, there is a smooth path $F_{s}$ in $U(M_{L}(C[0,1]))$ such that $\|F-\widetilde{F}\|<\varepsilon$ and $F_{1}(t)=\widetilde{F}_1(t)=\widetilde{W}(t)$ and $F_{s}(t)$ has no repeated eigenvalues for all $(s,t)\in[0,1]\times[0,1]$ and $|\text{length}(\widetilde{F})-\text{length}(F)|\leq\varepsilon$. By {Lemma \ref{sm}} and Remark \ref{remark}, there exist continuous maps $f^{1},f^{2},...,f^{L}:[0,1]\times[0,1]\longrightarrow S^{1}$ and unitaries $U_{s}(t)$ such that
$$F_{s}(t)=U_{s}(t)\text{diag}(f_{s}^{1}(t),f_{s}^{2}(t),...,f_{s}^{L}(t))U_{s}(t)^{*}.$$
Since $F_{1}(t)=\widetilde{W}(t)$ and $\widetilde{y}_{\alpha}=e^{-2\pi it(\frac{9}{10}-\varepsilon)}$, we can assume $f_{1}^{\alpha}=e^{-2\pi it(\frac{9}{10}-\varepsilon)}.$
Therefore, $f_{s}^{\alpha}$ is a path in $U(C[0,1])$ from a point near $1$ to $e^{-2\pi it(\frac{9}{10}-\varepsilon)}.$ By {Lemma \ref{ex0}},
$$\text{length}(f_{s}^{\alpha})\geq 2\pi(\frac{9}{10}-\varepsilon)-\varepsilon.$$
Therefore, $$\text{length}(\tilde{F}_s)\geq2\pi(\frac{9}{10}-\varepsilon)-\varepsilon-\varepsilon=2\pi(\frac{9}{10}-\varepsilon)-2\varepsilon,$$ and
$$cel(\varphi_{1,n}(u_\varepsilon))\geq 2\pi(\frac{9}{10}-\varepsilon)-5\varepsilon.$$\end{proof}

\begin{corollary}\label{co3} Let $A=\lim A_i$ and $u\in A_1$ be defined as in \ref{Ex2}, then  $\varphi_{1,\infty}(u)\in CU(A)$ and $$cel(\varphi_{1,\infty}(u))\geq 2\pi\cdot\frac{9}{10}.$$
\end{corollary}
\noindent Note that $CU(A)$ is the closure of the commutator subgroup of $U(A)$ (see \cite{lin3}, \cite{dHS}, \cite{Th}).  For $A=\text{M}_n(C[0,1])$ ($n\in\mathbb{N}$), $x\in CU(A)$ if and only if $\det(x(t))=1$ for each $t\in[0,1]$.
\begin{proof} For any $\varepsilon\in(0,\frac{1}{100})$,  let ${u_{\varepsilon}}\in A_1$ be defined as in Theorem \ref{th2}, then
 $$cel(\varphi_{1,n}(u_\varepsilon))\geq 2\pi(\frac{9}{10}-\varepsilon)-5\varepsilon
~~\mbox{for all } n\in \mathbb{N}.$$
Since $\|\varphi_{1,n}(u)-\varphi_{1,n}(u_\varepsilon)\|\leq 2\pi\varepsilon$ for all $n\in \mathbb{N}$, $$cel(\varphi_{1,n}(u))\geq cel(\varphi_{1,n}(u_\varepsilon))-2\pi\varepsilon\geq 2\pi(\frac{9}{10}-2\varepsilon)-5\varepsilon.$$
Since $\varepsilon$ is arbitrary, we have $cel(\varphi_{1,n}(u))\geq 2\pi\cdot\frac{9}{10}$, for all $n\in\mathbb{N}$. Therefore, $$cel(\varphi_{1,\infty}(u))\geq 2\pi\cdot\frac{9}{10}.$$\end{proof}

\begin{remark} Let $A_i,~\varphi_{i,i+1},~(i\in \mathbb{N})$ be defined as in \ref{Ex2}, for any $\varepsilon>0$, there exists $i$ such that $\frac{10^{k_i}-1}{10^{k_i}}\geq 1-\frac{\varepsilon}{2\pi}$. Let $u\in A_i$ be defined by
$$
u(t) =
  \begin{pmatrix}
   e^{-2\pi it\frac{10^{k_i}-1}{10^{k_i}}} & 0 & \cdots & 0 \\
   0 & e^{2\pi it\frac{1}{10^{k_i}}} & \cdots & 0 \\
   \vdots  & \vdots  & \ddots & \vdots  \\
   0 & 0 & \cdots & e^{2\pi it\frac{1}{10^{k_i}}}
  \end{pmatrix}_{10^{k_i}\times 10^{k_i}}.
$$ Then $\varphi_{i,\infty}(u)\in CU(A)$ and $$cel(\varphi_{i,\infty}(u))\geq 2\pi \frac{10^{k_i}-1}{10^{k_i}}\geq 2\pi(1-\frac{\varepsilon}{2\pi})=2\pi-\varepsilon.$$
\end{remark}
\begin{remark}
Our paper with above  results was first posted on arxiv on Dec. 2012. Later H. Lin posted a paper on Feb. 2013 (see \cite{lin2}). In his paper, he provides  examples with
$cel_{CU}(A)>\pi$ for unital simple $AH$-algebras $A$ with tracial rank one, whose $\text{K}_0$-group can realize all possible weakly unperforated Riesz group 
(see 5.11, 5.12 of \cite{lin2}).
He also obtained our theorem 3.11 with different method. But our example shows examples with $cel_{CU}(A)\geq2\pi$ for some simple $AH$-algebras (see 3.12, 3.15, 3.16). Of course, his paper contains many interesting results in the other directions.
%In H. Lin's paper, a similar example in $\text{M}_n(C[0,1])$ was given (see 5.1 in). Moreover, he proved $cel_{CU}(A)>\pi$ for unital simple $AH$-algebra $A$ with tracial rank one. But we give some examples to show that $cel_{CU}(A)>2\pi-\varepsilon$ for some unital somple $AH$-algebras $A$ with tracial rank one (see 3.12, 3.15, 3.16). Our paper was posted on arxiv first on Dec 2012 while Lin's was on Feb 2013 (with more results).
\end{remark}

%%%%%%%%%%%%%

\end{document}